\documentclass[11pt, one side,emlines]{amsart}
\usepackage{amssymb,latexsym,xy,eucal,mathrsfs}
\usepackage{MnSymbol}
\textwidth=15.2cm \textheight=24cm \theoremstyle{plain}
\newtheorem{lemma}{Lemma}[section]
\newtheorem{theorem}[lemma]{Theorem}
\newtheorem{proposition}[lemma]{Proposition}
\newtheorem{corollary}[lemma]{Corollary}

\theoremstyle{definition}
\newtheorem{example}[lemma]{Example}
\newtheorem{remark}[lemma]{Remark}

\newtheorem{definition}[lemma]{Definition}

\numberwithin{equation}{section}  \voffset
-55truept \hoffset -15truept
\begin{document}
\baselineskip 15truept
\title{ On Generalized Rickart $*$-rings}

\date{}

\author{ Anil Khairnar and Sanjay More }
	 
\address{\rm Department of Mathematics, MES Abasaheb Garware College, Pune-411004, India.}
 \email{\emph{ask.agc@mespune.in; anil\_maths2004@yahoo.com}}
 
  \address {\rm Department of Mathematics, Prof. Ramkrishna More College, Akurdi, Pune-411044, India.}
 \email{\emph{sanjaymore71@gmail.com}}


\maketitle {\bf \small Abstract:} {\small 
 A ring $R$ with an involution $*$ is a generalized Rickart $*$-ring if for all $x\in R$
the right annihilator of $x^n$ is generated by a projection for some positive integer $n$
depending on $x$. In this work, we introduce generalized right projection of an element in a $*$-ring and prove that every element in a generalized Rickart $*$-ring has generalized right projection. Various characterizations of generalized Rickart $*$-rings are obtained. We introduce the concept of  generalized weakly Rickart $*$-ring and provide a characterization of generalized Rickart $*$-rings in terms of weakly generalized Rickart $*$-rings. It is shown that generalized Rickart $*$-rings satisfy the parallelogram law. A sufficient condition is established for partial comparability in generalized Rickart $*$-rings. Furthermore, it is proved that pair of projections in a generalized Rickart $*$-ring possess orthogonal decomposition. }

\section{Introduction}

\indent   Kaplansky \cite{Kap} introduced Baer rings and Baer $*$-rings to
generalize various properties of $AW^*$-algebras (i.e., a $C^*$-algebra which is also a Baer $*$-ring), von Neumann
algebras and complete $*$-regular rings. The concept of a Baer $*$-ring arises naturally from the study of functional analysis. For instance, every von Neumann algebra is a Baer $*$-algebra. For recent work on rings with involution, one can refer to  \cite{Anil,Anil2,Anil3,Anil4,more1,Tha2}.

       A ring $R$ is said to be {\it reduced} if it does not contains
    nonzero nilpotent element.  A ring $R$ is said to be {\it abelian} if its every idempotent element is central.  
   Let $S$ be a nonempty subset of $R$. We write
          $r(S)= \{a \in R ~:~ s a = 0, ~\forall ~s \in S \}$, and is called the
          {\it right annihilator} of $S$ in $R$, and $l(S)= \{a \in R ~|~  a s = 0, ~\forall ~s \in S \}$, 
          is the
         { \it left annihilator} of $S$ in $R$. Let $R$ be a ring and $a\in R$, then we write $r(\{a\})=r(a)$ and $l(\{a\})=l(a)$.      
    A {\it $*$-ring $R$} is a ring equipped with an involution $x \rightarrow x^* $,
                        that is, an additive anti-automorphism of the period at most two.
    An element $e$ of a $*$-ring  $R$ is called a {\it projection} if it is self-adjoint 
         (i.e. $e=e^*$)
       and idempotent (i.e. $e^2=e$).  Let $P$ be a poset and $a,b \in P$. The join of $a$ and $b$, denoted by
       $a\vee b$, is defined as $a \vee b = \sup~ \{a, b\}$. The meet of $a$ and $b$, denoted by $a\wedge b$, is defined as $a \wedge b= inf~ \{a, b\}$.  
       In a poset $(P, \leq)$, $a < b$ denotes $a\leq b$ with $a \neq b$. Let $R$ be a $*$-ring and $e,f \in R $ be projections, we say that $e \leq f$ if $e=ef$, this defines a partial order on the set of all projections in $R$.
      A $*$-ring $R$ is said to be a {\it  Rickart $*$-ring},
  if for each $x \in R$, $r(x)=eR$, where $e$ is a projection
  in $R$. 
   For each element $x$ in a Rickart $*$-ring, there is unique projection
  $e$ such that $xe=x$ and $xy=0$ if and only if $ey=0$, called the {\it
  	right projection} of $x$, denoted by $RP(x)$. Similarly, the left
  projection $LP(x)$ is defined for each element $x$ in Rickart $*$-ring.
  A $*$-ring $R$ is said to be a {\it weakly Rickart  $*$-ring}, if for any $x \in R$, there exists a projection $e$
  such that (1) $xe=x$, and (2) if for $y\in R$, $xy=0$ then $ey=0$.
  
     In \cite{Ber}, Berberian posed the following open problem.
            
 {\bf Problem 1:} Can every weakly Rickart $*$-ring be
 embedded in a Rickart $*$-ring with preservation of $RP$'s?\\
 In \cite{Ber} Berberian provided a partial solution to this problem. Subsequently, in \cite{Tha3}, authors offered another partial solution. In \cite{more1}, a more general partial solution to Problem 1 is presented. 
 
  Let $R$ be a $*$-ring. The projections $e, f$ in $R$ are said to be {\it equivalent}  (written as $e \sim f$),  if there exists $w\in R$
 such that $w^*w=e$ and $ww^*=f$ (see \cite {Ber}). By \cite[Proposition 7, page 5]{Ber},
 the relation $\sim$ is an equivalence relation on the set of projections in $R$.   A $*$- ring is said to satisfy {\it parallelogram law} if  $e-e\wedge f \sim  e\vee f-f$ for every pair of projection $e$ and $f$ whose $e\wedge f$ and $e\vee f$ exists. Projections $e$ and $f$ in a $*$-ring $R$ are said to be {\it partially comparable} if there exist non-zero projections $e_{0}$ and $f_{0}$ such that $e_{0} \le  e$,  $f_{0} \le f$ and $e_{0} \sim  f_{0}$.
 We say that a $*$-ring $R$ has  {\it partial comparability} $(PC)$ if $eRf \neq  0$ implies  $e$ and  $f$ are partially comparable.  Let $R$ be a $*$-ring and $e, f$ be projections in $R$. We say that $e$ is {\it dominated } by $f$,
 written as $e \lesssim f$, if $e \sim g \leq f$ for some projection $g \in R$. Projections $e$ and $f$ in a $*$-ring $R$ are said to be {\it generalized comparable} if there exists central projection $h$ such that $he \lesssim hf$ and $(1-h)f \lesssim (1-h)e$.
 We say that $R$ has {\it generalized comparability} $(GC)$ if every pair of projection in $R$ is generalized comparable. We say that a $*$-ring $R$ has { \it orthogonal GC} if every pair of orthogonal projections are generalized comparable. 
 Projections $e$ and $f$ in a $*$-ring $R$ are said to be {\it very orthogonal} if there exists central projection $h$ such that $he = e$ and $hf = 0$.
  
In \cite{cui2018pi}, the authors introduced the concept of generalized Rickart $*$-ring.  A $*$-ring $R$ is called a generalized Rickart $*$-ring, if, for any $x\in R$,
 there exists a positive integer $n$ such that $r(x^n)=gR$, for some projection $g$ of $R$.
 In generalized Rickart $*$-rings, we also have $l(x^n)=Rh$ for some projection $h\in R$. This indicates that generalized Rickart $*$-rings are left-right symmetric. Generalized Rickart $*$-rings serve as a common generalization of both Rickart $*$-rings and generalized Baer $*$-rings. In \cite{ahmadigen2021}, M. Ahmadi and A. Moussavi explored the behavior of the generalized Rickart $*$-condition under various constructions and extensions. They also provided  examples of
 generalized Rickart $*$-rings and identified classes of finite and infinite-dimensional Banach $*$-algebras that are generalized Rickart $*$-rings but not Rickart $*$-rings. 
 
    In the second section of this paper, we introduce generalized right projection of an element in a $*$-ring and prove that every element of a generalized Rickart $*$-ring has a generalized right projection. Properties of generalized right projection of  elements in a generalized Rickart $*$-ring are also studied. We introduce the concept of a generalized weakly Rickart $*$-ring and provide a characterization of generalized Rickart $*$-rings in terms of generalized weakly Rickart $*$-rings. It is shown that the center and corner of a generalized weakly Rickart $*$-ring are themselves generalized weakly Rickart $*$-ring. In section 3, we pose a problem for generalized Rickart $*$-rings analogous to Problem 1 for Rickart $*$-rings and provide a partial solution. In section 4, we prove that generalized Rickart $*$-rings satisfy the parallelogram law. A sufficient condition is established for generalized Rickart $*$-rings to exhibit partial comparability. Furthermore, it is shown that pairs of projections in a generalized Rickart $*$-ring that satisfying the parallelogram law possess orthogonal decomposition. 
  
 \section{Weakly Generalized  Rickart $*$-rings}  
 
   It is evident that every Rickart $*$-ring is a generalized Rickart $*$-ring. 
  In this section, we first recall examples and results from \cite{ahmadigen2021}, which provides instances of generalized Rickart $*$-rings that are not Rickart $*$-rings. We then introduce generalized right projection of an element and generalized left projection of an element in a $*$-ring. We prove that every element of a generalized Rickart$*$-ring has a generalized right projection and discuss the properties of these projections. Furthermore, we introduce the class of generalized weakly Rickart $*$-rings and establish characterizations of generalized Rickart $*$-rings. 
  
 \begin{example} [\cite{ahmadigen2021}, Example 2.8]
 	(i) Let $S = C[x, y]/(x, y)^n$. Then $S$ is a commutative local ring with unique maximal ideal $(\overline x, \overline y)$
 	having index of nilpotency $n$. The set $S$ with the conjugate map as the involution, is a generalized Rickart $*$-ring but not Rickart $*$-ring.\\
 	(ii) Let $G$ be a finite abelian $p$-group. Then the group algebra $T = F_pG$ is a finite commutative local
 	ring with unique maximal ideal $rad(T)$ having index of nilpotency $|G|$. Let $*$ be the involution on the group ring $T$ defined by $(\sum a_gg)^* = \sum a_gg^{-1}$. Hence $T$ is a generalized Rickart $*$-ring but not Rickart $*$-ring.\\
 	(iii) Let $S$ and $T$ be the rings in (i) and (ii) respectively. Let $n$ be a positive integer and let $p$ be 	a prime. Then the $*$-ring $R = S \oplus T \oplus \mathbb Z_{p^n}$ is a generalized Rickart $*$-ring that is not Rickart $*$-ring.\\
 	(iv) Let $T=\{a_0+a_1i+a_2j+a_3k ~:~ a_i\in \mathbb Z_2~\text{for}~ i=1,2,3,4\}$ be the Hamilton quaternions over $\mathbb Z_2$. Then $T$ is a commutative ring such that every element of $T$ is either invertible or nilpotent.
 	For $\alpha=a_0+a_1i+a_2j+a_3k$ with $a_0,a_1,a_2,a_3\in \mathbb Z_2$, define $\alpha^*=a_0-a_1i-a_2j-a_3k$. Then $*$ is an involution for $T$. Then $T$ is generalized Rickart $*$-ring but not a Rickart $*$ ring.
 \end{example}
 
  \begin{definition}
  	(\cite{ahmadiquasi2020}, Definition 5.1). Let $R$ be a ring with unity, and let $n \geq 2$ be an integer. Put $V_n = \sum_{i=1}^{n-1}E_{i,i+1}$.
  	The triangular matrix rings $S_n(R), A_n(R), B_n(R), U_n(R)$, and $V_n(R)$ are define as follows:
  	$$A_n(R)=RI_n+\displaystyle \sum_{l=2}^{[\frac{n}{2}]}RV_n^{l-1}+\sum_{i=1}^{[\frac{n+1}{2}]}\sum_{j=[\frac{n}{2}]+i}^{n}RE_{ij}.$$
  	$$B_n(R)=RI_n+\displaystyle \sum_{l=3}^{[\frac{n}{2}]}RV_n^{l-2}+\sum_{i=1}^{[\frac{n+1}{2}]+1}\sum_{j=[\frac{n}{2}]+i-1}^{n}RE_{ij}.$$
  	$$U_n(R)=RI_n+\displaystyle \sum_{i=1}^{[\frac{n-1}{2}]}\sum_{j=[\frac{n}{2}]+1}^{n}RE_{ij}+\sum_{j=[\frac{n-1}{2}]+2}^{n}RE_{[\frac{n-1}{2}]+1,j}.$$
  	$$S_n(R)=RI_n+\displaystyle \sum_{i<j}RE_{i,j}, ~~ V_n(R)=RI_n+\displaystyle \sum_{l=2}^nRV_n^{l-1}.$$
  \end{definition}
  
  \begin{proposition} [\cite{ahmadigen2021}, Theorem 4.6] 
  	Let $A$ be an abelian $C^*$-algebra. If $A$ is a generalized Rickart $*$-ring, then the Banach $*$-algebras $S_n(\mathcal A), A_n(\mathcal A), B_n(\mathcal A), U_n(\mathcal A)$, and $V_n(\mathcal A)$ are generalized Rickart $*$-rings but not Rickart
  	$*$-rings. In particular, the Banach $*$-algebras $S_n(C), A_n(C), B_n(C), U_n(C)$, and $V_n(C)$ are generalized
  	Rickart $*$-rings but not Rickart $*$-rings.
  \end{proposition}
  
   \begin{example}[\cite{ahmadigen2021}, Example 4.7]
  	Let $A$ be an infinite-dimensional commutative von Neumann algebra. Then $A$ is a Baer $*$-ring (and hence a Rickart $*$-ring). Now let $R$ be the ring $Sn(C) (An(C), Bn(C), Un(C), or Vn(C))$ and put $T = A\oplus R$ (equipped with the max norm). Then $T$ is an infinite-dimensional Banach $*$-algebra and $T$ is a generalized Rickart $*$-ring but not a Rickart $*$-ring.
  \end{example}
 
   Following is an example of a finite generalized Rickart $*$-ring that is not a Rickart $*$-ring.
  
\begin{example}	Let $ R=\mathbb{Z}_4 $ be a $*$-ring with an identity involution .
 Here $0, 1$ are only projections in $ \mathbb{Z}_4 $.
 Now $ r(2)=\{0,2\} \ne 0\mathbb{Z}_4$ and $  r(2)=\{0,2\} \ne 1\mathbb{Z}_4 $.
 Hence $ r(2)$ is not generated by any projection in $R$. Therefore $ \mathbb{Z}_4 $ is not a Rickart $ *$-ring. Observe that $ r(0)=1 \mathbb{Z}_4, r(1)=0 \mathbb{Z}_4 $,
 $ r(2^2)=r(0)=1\mathbb{Z}_4 $, and  $r(3)=0 \mathbb{Z}_4$. 
 Hence for every $ x\in \mathbb{Z}_4,\exists~ n\in\mathbb{N} $ and projection $e$ in $ \mathbb{Z}_4 $ such that $ r(x^n)=e\mathbb{Z}_4 $. Therefore  $ \mathbb{Z}_4$ is a generalized Rickart $*$-ring.
 \end{example}
   
The following is an example of a finite $*$-ring that is not a generalized Rickart $*$-ring.
 
 \begin{example}  
 Let $ R=M_4(\mathbb{Z}_4) $ with transpose as an involution. Then $R$ is not a generalized Rickart $*$-ring (see Example \ref{s2ex2}). 
\end{example}

 In the following result, we prove that generalized Rickart $*$-rings, always contain the multiplicative identity (unity) element.
 
\begin{proposition} \label{s2pr1} If $R$ is a generalized  Rickart $ *$-ring then it has unity element.
\end{proposition}
\begin{proof} Let $R$ be a generalized  Rickart $ *$-ring.
	Since for any $ n\in\mathbb{N} $, $r(0^n)=r(0)=R=eR $ for some projection $e \in R$. 
	Therefore for $x\in R$, we have $x=ey$ for some $y \in R$. This implies $ex=e^2y=ey=x$.
	Thus $e$ is the unity element in $R$.
\end{proof}

\begin{proposition} \label{s2pr2} If $R$ is a generalized  Rickart $*$-ring then for every $ x \in R $ there exists $ n\in\mathbb{N} $ and a projection $ e\in R$ such that $x^ne=x^n$; and for $y \in R$, if $x^ny=0$ then $ey=0 $.
\end{proposition}		 
\begin{proof} As $R$ is a  generalized  Rickart $ *$-ring,
	for every $ x\in R, r(x^n)=gR $ for some $ n\in\mathbb{N}$ and for some projection  $g\in R $.
	Let $ e=1-g $, then $x^ne=x^n(1-g)=x^n-x^ng=x^n$. Let $y\in R$ and $ x^ny=0 $. 
  Then	$y\in r(x^n)=gR$, this gives $ y=gy$, that is $(1-g)y=0$. Therefore $ey=0$.
	Thus $ x^ny=0$ implies that $ey=0 $.
\end{proof}

Observe that in the above proposition, the projection $e$ is the smallest projection such that  $ x^ne=x^n$ (see Proposition \ref{s2pr3}).
Now, we introduce the generalized right projection of an element in a $*$-ring.
\begin{definition}  Let $R$ be a $*$-ring and $ x \in R $. The projection $ e \in R $ is said to be generalized right projection of $x$ denoted by $GRP(x)$ if there exists $ n\in\mathbb{N} $ such that  $ x^ne=x^n$; and for $y\in R$ if $x^ny=0$ then $ey=0$. 
	\end{definition}
  Similarly, we introduce the generalized left projection of an element in a $*$-ring.
 \begin{definition}  Let $R$ be a $*$-ring and $ x \in R $. The projection $ f \in R $ is said to be generalized left projection of $x$ denoted by $GLP(x)$ if there exists $ n\in\mathbb{N} $ such that  $ fx^n=x^n$; and for $y\in R$ if $yx^n=0$ then $yf=0$. 
 \end{definition} 
  
 \begin{remark}  \label{s2rm1}	By Proposition \ref{s2pr2}, every element of a generalized Rickart $*$-ring possesses a generalized right projection.
   Similarly, every element of a generalized Rickart $*$-ring possesses a generalized left projection. 	
  \end{remark}
  
\begin{definition}  A $*$-ring $R$ is said to be generalized weakly Rickart $*$-ring  if every $x\in R$ possesses a generalized right projection. That is $GRP(x)$ exists for every $ x \in R $.
\end{definition}
     By Remark \ref{s2rm1}, every generalized Rickart $*$-ring is a generalized weakly Rickart $*$-ring. The following is an example of a $*$-ring which is not a weakly generalized Rickart $*$-ring.
   
\begin{example} \label{s2ex2} Let $R=M_4(\mathbb{Z}_4)$ and $ A=\begin{bmatrix}
		1&1&1&1&\\
		0&0&0&0&\\
		0&0&0&0&\\
		0&0&0&0&
	\end{bmatrix} $, 
	$ B=\begin{bmatrix}
		1&1&1&1&\\
		1&1&1&1&\\
		1&1&1&1&\\
		1&1&1&1&\\
	\end{bmatrix}  \in R$.\\	
Suppose $GRP(A)=E$. Then there exits $n \in \mathbb N$, such that $A^nE=A^n$; and $A^nB=0$ implies $EB=0$. 
That is  $AE=A$; and $AB=0$ implies $EB=0$. 
But $AE=A$ implies $e_{11}+e_{21}+e_{31}+e_{41}=1$. This gives $e_{11}+e_{12}+e_{13}+e_{14} =1 $, a contradiction.
Thus $ M_4(\mathbb{Z}_4) $ is not a generalized weakly Rickart $*$-ring. 
\end{example}

\begin{proposition} \label{s2pr3} Let $R$ be a generalized  Rickart $*$-ring.  
	\begin{enumerate}
		\item   For $x,y \in R$, if $ xy=0 $ then $GRP(x)GLP(y)=0$.
		\item  If $GRP(x)=e$ then $e$ is the smallest projection such that $x^ne=x^n$ for some $n \in \mathbb N$.
	\end{enumerate}
\end{proposition}
\begin{proof} (1):  Let $x,y\in R $ and $GRP(x)=e,~ GLP(y)=f$.
  Then for some $n\in\mathbb{N} $, $x^ne=x^n$; and $x^nz=0$ implies $ez=0 $.
	Also, for some $m\in\mathbb{N}$, $fy^m=y^m$; and $zy^m=0$ implies $zf=0$.
	Suppose $xy=0$. Then $x^{n-1} xy=0$ implies $x^ny=0$. This gives $ey=0$.
	Therefore $eyy^{m-1}=0$ implies $ey^m=0$. Hence $ef=0$. That is $GRP(x) GLP(y)=0$.\\
	(2): Let $f$ be a projection such that $x^nf=x^n$. Then $x^n(e-f)=x^ne-x^nf=0$. This gives 
	$e(e-f)=0$, that is $e=ef$. Hence $e\leq f$.
	Therefore $e$ is the smallest projection such that $x^ne=x^n$. 
 \end{proof}
 	
    The converse of Proposition \ref{s2pr3} (1) is not true. 
    
    \begin{example} Let $R=\mathbb{Z}_{12} $. Observe that $0, 1, 4, 9$ are the only projections in $R$.
	Let $ x=2,~y=3 \in R$. Since $2^2\cdot  4=4=2^2$, we have $GRP(2)=4$. 
	Also, $9\cdot 3^2=9=3^2$, gives $GLP(3)=9$. Therefore $GRP (2) GLP (3)=4 \cdot 9=0$, but $2 \cdot3=6 \ne 0$. Thus $GRP(x)~GLP(y)=0$ but $xy\ne0$. 
  \end{example}

In the following result, we provide a condition under which a $*$-subring of a generalized Rickart $*$-ring becomes a generalized Rickart $*$-ring.

\begin{proposition} \label{s2pr4} 	Let $R$ be a generalized  Rickart $ *$-ring and $B$ be a $*$-subring of $R$        satisfying the following conditions, 
	\begin{enumerate}
		\item $B$ has unity element 
		\item if $x\in B$ then $GRP(x)\in B$.
	\end{enumerate}
	Then $B$ is a generalized  Rickart $*$-ring.
\end{proposition}
\begin{proof} Let $x\in B$ and $GRP(x)=e$. Therefore for some $ n\in\mathbb N $, $x^ne=x^n$; and $x^ny=0$ implies $ey=0$.	Since $e\in B$. for $y\in B,~ x^ny=0$ if and only if $ey=0$ if and only if $y=y-ey$ if and only if 
	$y=(1-e)y$. Therefore $r(x^n)=(1-e)B$. Hence $B$ is a generalized  Rickart $*$-ring. 
\end{proof}

Let $R$ be a ring and $S$ be a nonempty subset of $R$, the
commutant of $S$ in $R$, denoted $S'$, is the set of elements of $R$ that commute
with every element of $S$, that is $S'=\{ x \in R ~:~ xs=sx,$ for all $s\in S\}$.
	We write $S''=(S')'$, called the bicommutant of $S$.
	
	In the following result, we provide a condition under which a $*$-subring of a generalized Rickart $*$-ring becomes a generalized Rickart $*$-ring, and generalized right projection of every element in the $*$-subring remains within it. 
	
\begin{proposition} \label{s2pr5} Let $R$ be a generalized  Rickart $*$-ring and 
	$B$ be a $*$-subring of $R$ such that $ B=B''$. Then 
	\begin{enumerate}
		\item  $x\in B$ implies $GRP(x)\in B$.
		\item  $B$ is a generalized  Rickart $*$-ring.
	\end{enumerate}
\end{proposition}
\begin{proof} (1): Suppose $x\in B$ thus $x\in R$. Let $ GRP(x)=e$. Therefore there exists $n \in \mathbb N$ such that $x^ne=x^n$; and for $y\in R$ $x^ny=0$ implies $ey=0$. To prove $e\in B=B''=(B')'$, it is enough to show $ey=ye$ for all $y\in B'$.
	Now $xy=yx $  gives $x^ny=yx^n$. Therefore $x^n(y-ye)=x^ny-x^nye=x^ny-yx^ne=x^ny-yx^n=0$.
	This implies $e(y-ye)=0$ and hence $ey=eye$. Replace $y$ by $ y^*$. So $ey^*=ey^*e$, that is $ye=eye=ey$.
	Therefore $ey=ye$ for all $y\in B'$. Hence  $e\in B$, that is $GRP(x)\in B$.\\
    (2): We know for any nonempty subset $S$ of a ring $R$, $1\in S'$ and $ S'= S'''$.
    Since $B=B''$ is equivalent to $B=S'$ for some $*$-subset $S$ of $R$, we have $1\in B$.	
	By (1), $x\in B$ implies $GRP(x)\in B $. By Proposition \ref{s2pr4}, $B$ is a generalized  Rickart $ *$-ring.
\end{proof}

\begin{proposition} \label{s2pr6} Let $R$ be a generalized  Rickart $ *$-ring and $x\in R$. Then there exists $n\in\mathbb N$ such that $r(x^n)=(1-GRP(x))R$.
\end{proposition}
\begin{proof}  Let $x\in R$ and $GRP(x)=e$. Therefore there exist $n \in \mathbb N$ such that 
	$x^ne=x^n$; and for $z \in R$, $x^nz=0$ implies $ez=0$.
	We prove that $r(x^n)=(1-e)R$. Let $y\in r(x^n)$, then $x^ny=0$, and hence $ey=0$. 
	Therefore $y=y-ey=(1-e)y\in(1-e)R$. Hence $r(x^n)\subseteq  (1-e) R$. 
	Let $w\in(1-e)R$, then $w=(1-e)w=w-ew$. Hence $ew=0$ implies $x^new=0$. Which gives $x^nw=0$, and hence  $w\in r(x^n)$. Therefore $(1-e)R\subseteq r(x^n)$. Thus $r(x^n)=(1-e)R=(1-GRP(x))R$.
\end{proof}

In the following result, we provide a characterization of a generalized Rickart $*$-ring 

\begin{proposition} \label{s2pr7} A $*$-ring $R$ is generalized Rickart $*$-ring if and only if $R$ has unity element and
		for each $x\in R$ there exists a projection $e$ such that $r(x^n)=r(e)$ for some $n\in \mathbb N$.	
\end{proposition}
\begin{proof} Suppose $R$ is a generalized  Rickart $*$-ring.
Let $x\in R$ and $GRP(x)=e$. Therefore for some $ n\in\mathbb N $, $x^ne=x^n$; and $x^ny=0$ if and only if  $ey=0$. Hence $y\in  r(x^n)$ if and only if $y\in r(e)$ for some $n\in \mathbb N$. Thus $r(x^n)=r(e)$. 
	Conversely, suppose $R$ has unity element and
	for each $x\in R$ there exists a projection $e$ such that $r(x^n)=r(e)$ for some $n\in \mathbb N$.
	Therefore $r(x^n)=r(e)=(1-e)R$ for some $n\in \mathbb N$.
	Therefore for any $x\in R$, there exists a projection $1-e$ such that  
	$r(x^n)=(1-e)R$ for some $n\in \mathbb N$. Hence $R$ is a generalized  Rickart $*$-ring. 
\end{proof}

 Following is a characterization of a generalized Rickart $*$-ring in terms of generalized weakly Rickart $*$-rings.

\begin{proposition} \label{s2pr8} The following conditions on a $*$-ring $R$ are equivalent.
	\begin{enumerate}
		\item $R$ is a generalized  Rickart $*$-ring.  
		\item $R$ is a generalized weakly Rickart $*$-ring with unity.
	\end{enumerate}
\end{proposition}
\begin{proof} Suppose $R$ is a generalized Rickart $*$-ring.
	By Proposition \ref{s2pr1}, $R$ has unity element. Also, by Proposition \ref{s2pr2}, $R$ is a generalized weakly Rickart $*$-ring. Therefore $R$ is a generalized weakly Rickart $*$-ring with unity.
	Conversely suppose $R$ is a generalized weakly Rickart $*$-ring with unity.
	Let $x\in R$. Then $GRP(x)$ exists in $R$. Therefore $r(x^n)=(1-GRP(x)) R$ for some $n \in \mathbb N$.
	Hence $R$ is a generalized  Rickart $*$-ring.
\end{proof}

Following is the example of weakly generalized Rickart $*$-ring which is not generalized Rickart $*$-ring.

\begin{example}
Let $R=\{x=(x_1,x_2,\cdots )~|~ x_i\in \mathbb{C} ,~ for ~each ~x,~ there ~ exists~ m\in \mathbb{N} ~such ~that~ x_k=0 ~for ~all~ k >  m\}$.
Then $R$ is a ring with component-wise operations multiplication and addition. Define the involution $*$ on $R$ as, for $x=(x_1,x_2,\cdots )\in R,~ x^*=(\bar{x_1},\bar{x_2},\cdots )$. Clearly $e=(e_1,e_2,\cdots )\in R $ is a projection if $e_i\in \{0,1\}$. Also, the unity element is $u=(1,1,\cdots )$ and $u\notin R$ . Let $x=(x_1,x_2,\cdots )\in R$. Then there  exists $m\in \mathbb{N}$ such that $x_k=0$ for all $k >  m$. Let us find $n\in \mathbb{N}$ and projection $e \in R$ such that $x^ne=x^n$ and $x^ny=0$ implies $ey=0$. Choose $n=1$ and define $e=(e_1,e_2,\cdots )$ as $e_i=1$ if $x_i\neq 0$ and $e_i=0$ if $x_i= 0$. As $x$ has only finitely many non-zero components, $e$ also has finitely many non-zero components. So  $e \in R$. Further, $e_i\in \{0,1\}$ and $e_i=\bar{e_i}$, thus $e$ is a projection in $R$. For each component $(xe)_i=x_ie_i$. If $x_i\neq 0$ then $e_i=1$, so $(xe)_i=x_i\cdot 1=x_i$. If $x_i= 0$ then $e_i=0$, so $(xe)_i=0=x_i$. Thus, $xe=x$. Now suppose $xy=0$. This means $x_iy_i=0$ for all $i$. If $x_i\neq 0$ then $y_i$ must be 0. If $x_i= 0$ then $y_i$ can be anything. We have $(ey)_i=e_iy_i$. Thus, If $x_i\neq 0$ then $e_i=1$. so $(ey)_i= 1\cdot 0=0$. If $x_i= 0$ then $e_i=0$ and hence $(ey)_i= 0$. Therefore, $ey=0$. Thus $R$ is a weakly generalized Rickart $*$-ring.
 To show $R$ is not a generalized Rickart $*$-ring. We will find $x\in R$ such that for any $n\in \mathbb{N},~ r(x^n)\neq eR$ for any projection $e\in R$. Let $x=(0,1,0,0, \cdots)\in R$. We have $x^n=x$ for any $n\geq 1$. Let us find $r(x)$. Suppose $y=(y_1,y_2, \cdots )\in r(x)$. Hence, $xy=0$. Therefore $(0,1,0,0, \cdots)\cdot (y_1,y_2, \cdots )=(0,0, \cdots )$. This gives $y_2=0$. Thus, $r(x)=\{(y_1,0,y_3,y_4, \cdots )~|~y_i\in \mathbb{C}\}$. Suppose  $r(x)=eR$ for some projection $e \in R$. Let $e_k=0$ for $k>n_0$. If $y\in r(x)$, then  $y\in eR$, thus $y_i=e_iz_i$ for some $z=(z_1,z_2, \cdots )\in R$. As $y_2=0$ for any $y\in r(x)$. So $(ez)_2=e_2z_2=0$. Hence $e_2=0$. Also, for $w=(1,0,0,\cdots )\in r(x)$, we have $e_1=1$. Similarly, $e_3=1$. Therefore $e=(1,0,1,1,1, \cdots )$. This contradicts the fact that $e$ has only finitely many non-zero components. Hence,  $r(x^n)\neq eR$ for any projection $e\in R$  and for any $n\in \mathbb{N}$. Thus, $R$ is not a generalized Rickart $*$-ring.
\end{example}

In the following result, we prove that the center of a generalized weakly Rickart $*$-ring is itself a generalized weakly Rickart $*$-ring.
 
\begin{proposition} \label{s2pr9} The center of a generalized weakly Rickart $*$-ring is generalized weakly Rickart $*$-ring.
\end{proposition}
\begin{proof} Suppose $R$ is a generalized weakly Rickart $ *$-ring.
	Let $C(R)$ denote the center of $R$ and $x \in C(R) $.
	We will prove that $GRP(x)$ exists in $C(R)$.
	Since $x\in R$ and $R$ is a generalized weakly Rickart $*$-ring. 
	Therefore $GRP(x)=e$ exists in $R$. That is there exist $n \in \mathbb N$ such that $x^ne=x^n$; and $x^ny=0$ implies $ey=0 $. Hence there exist $n \in \mathbb N$ such that $e(x^n)^* =(x^n)^* $; and $y^*(x^n)^*=0$ implies $y^*e=0$.
	Since $x \in c(R)$, we have $x^n(r-re)=x^nr-x^nre=x^nr-rx^ne=x^nr-rx^n
	=x^nr-x^nr=0$. Therefore $e(r-re)=0$, that is $er-ere=0$, which gives $er=ere$.
	Also, $(r-er)(x^n)^*=r(x^n)^*-er(x^n)^*=r(x^n)^*-e(x^n)^*r=r(x^n)^*-(x^n)^*r=r(x^n)^*-r(x^n)^*=0$.
	Hence $(r-er)e=0$ implies $re=ere$.	Therefore $er=re$ for all $ r\in R $.
	Hence $e\in C(R)$, that is $ GRP(x)=e\in C(R)$. 
	Thus $C(R) $ is a generalized weakly Rickart $*$-ring.
\end{proof}	

  The involution $*$ of a ring $R$ is called {\it weakly proper} if for any $ x\in R $, $ xx^*=0$ implies $x^n=0$ for some $ n\in\mathbb N$.

 \begin{proposition} \label{s2pr10} Let $R$ be a generalized weakly Rickart $*$-ring. Then,\\
	\begin{enumerate}
		\item for each $ x\in R$ there exist $n \in \mathbb N$ such that $r(x^n) \cap(x^*)^nR=\{0\}$.
		\item the involution $*$ is weakly proper.
	\end{enumerate}
\end{proposition}
\begin{proof} (1): Let $x \in R$. Since $R$ is a generalized weakly Rickart $*$-ring, 
	 there exists $n\in\mathbb N$ and a projection $e \in R$ such that $x^ne=x^n$; and for $y \in R$, $x^ny=0$ implies $ey=0$. We prove that $r(x^n) \cap (x^*)^nR=\{0\}$.
	Let $y\in r(x^n)\cap (x^*)^nR$.
	Therefore $x^ny=0$ and $y=(x^*)^ns$ for some $ s\in R$.
	This gives $y=(x^*)^ns=(x^n)^*s=(x^ne)^*s=e(x^n)^*s=e(x^*)^ns=ey=0$.
	Hence $r(x^n)\cap (x^*)^nR=\{0\}$. \\
	(2): Let $ xx^*=0$. Therefore	
	 $x^n(x^*)^n=0$. This gives $e(x^*)^n=0$.
	Hence $e(x^n)^*=0$ implies $(x^ne)^*=0$, and this gives $(x^n)^*=0$. Therefore $x^n=0$.
\end{proof}

\begin{corollary}[\cite{ahmadigen2021}, Proposition 2.11] Let $R$ be a generalized Rickart $*$-ring. Then\\
	(i) for each $x\in R$, there exists an integer $n \geq 1$ such that $r(x^n) \cap (x^*)^nR = 0 $;
	(ii) the involution $*$ is weakly proper.
\end{corollary}

The following result provides the characterization of a generalized weakly Rickart $*$-ring.

\begin{proposition} \label{s2pr11} The following conditions on a $*$-ring $R$ are equivalent.
	\begin{enumerate}
		\item $R$ is generalized weakly Rickart $*$-ring. 
		\item Involution $*$ is weakly proper and for every $ x\in R$ there exist
		 $n\in\mathbb N$ and a projection $e$ in $R$ such that $r(x^n)=r(e)$.		
	\end{enumerate}
\end{proposition}
\begin{proof} Suppose $R$ is a generalized weakly Rickart $ *$-ring.
	By Proposition \ref{s2pr10}, involution on $R$ is weakly proper.
	Let $ x\in R $ and  $GRP(x)=e$.	Let $y\in r(x^n)$.
	Therefore $x^ny=0$ implies $ey=0$, and hence $y\in r(e)$. Thus $r(x^n)\subseteq r(e)$.
	Now let $ y\in r(e)$. Therefore $ey=0$, which implies $x^ny=x^ney=x^n0=0$.
	Hence $y\in r(x^n)$. Therefore $r(e) \subseteq r(x^n)$.	Thus $r(x^n)=r(e)$.
    Conversely, suppose involution $*$ is weakly proper and for every $ x\in R$ there exist
    $n\in\mathbb N$ and a projection $e$ in $R$ such that $r(x^n)=r(e)$.
	Let $x\in R$. Since $r(x^n)=r(e)$, we have $e(1-e)=0$ implies $1-e\in r(x^n)$. 
	Therefore  $x^n(1-e)=0$ implies $x^n=x^ne$.
	If $x^ny=0$ then $y \in r(x^n)$. This gives $y\in r(e)$ and hence $ey=0$.
	Therefore $e=GRP(x)$ and $R$ is a generalized weakly Rickart $*$-ring.
\end{proof}

  In the following result, we prove that the corner of a generalized weakly Rickart $*$-ring is itself a generalized weakly Rickart $*$-ring.
  
\begin{proposition} \label{s2pr12} Let $R$ be a $*$-ring  and $e$ be a projection in $R$. If $R$ is a generalized weakly Rickart $*$-ring then so is $eRe$.
\end{proposition}
\begin{proof} Let $x\in eRe$. Then $x \in R $. Since $R$ is a generalized weakly Rickart $*$-ring,
	$GRP(x)=f$ exists in $R$.
	Therefore there exist $n \in \mathbb N$ such that $x^nf=x^n$; and for $y\in R$, $x^ny=0$ implies $fy=0$.
	Since $x\in eRe$, we have $x=exe$. So $x^ne=(exe)^ne=(exe)^n=x^n$.Therefore $x^n(e-f) =x^ne-x^nf=x^ne-x^n$. 
	 Hence $x^n(e-f) =0$, which implies $f(e-f)=0$ thus $fe=f$.
	 We prove that $GRP(x)=f$ in $eRe$.
    As above $x^nf=x^n$. 	
	 Suppose $x^n(eze)=0$, then $feze=0$. Note that $f=ef=efe\in eRe$.
     Therefore  $GRP(x)=f$ in $eRe$.
	Thus $eRe$ is a generalized weakly Rickart $*$-ring.
\end{proof}

\begin{proposition} \label{s2pr13} Let R be a generalized weakly Rickart $*$-ring and $S$ be a self-adjoint subset of $R$ and $x\in S'$. If $GRP(x)=e$ then $ se=ese=es $ for all $ s\in S $.
\end{proposition}
\begin{proof}  Since $x\in S'$, we have $xs=sx$ for all $s\in S$.
	As $GRP(x)=e$. Then there exists  $ n\in\mathbb N$ such that  $x^ne=x^n$; and for $y \in R$, $x^ny=0$ implies $ey=0$. We have $x^ns=sx^n$. 
	Now $x^n(se-es)=x^nse-x^nes=sx^ne-x^ns=sx^n-x^ns=0$. Therefore $e(se-es)=0$ implies $ese=es$.
	Replacing $s$ by $s^*$ we get $es^*e=es^*$.
	Therefore  $ese=se$.
    Thus  $se=ese=es$.
\end{proof}

\begin{lemma} If $R$ is a generalized weakly Rickart $ *$-ring and $S$ is self adjoint subset of $R$, then $ S' $ is a weakly generalized Rickart $ *$-ring.
\end{lemma}
\begin{proof} Let $ x\in S' $ and $GRP(x)=e $ in $R$.
	By Proposition \ref{s2pr13}, $se=es$ for all $ s\in S$.
	Hence $e\in S'$.
	Therefore  $GRP(x)=e $ in $S'$. Thus $S'$ is a generalized weakly Rickart $*$-ring.
\end{proof}

 \section{Unitification of generalized weakly Rickart $*$-ring}

 Recall the definition of unitification of a $*$-ring given by Berberian \cite{Ber}.
 Let $R$ be a $*$-ring. We say that $R_1$ is a unitification of $R$, if there exists a  ring $K$, such that,\\
 1) $K$ is an integral domain with involution (necessarily
 proper), that is, $K$ is a commutative $*$-ring with unity
 and without divisors of zero (the identity involution is
 permitted),\\
 2) $R$ is a $*$-algebra over $K$ (i.e., $R$ is a left $K$-module such that,
 identically
 $1a=a,~\lambda (ab)=(\lambda a)b=a (\lambda b),~ and~(\lambda a)^*=\lambda ^*
 a^*$, for $\lambda \in K$ and $a,b \in R$).\\
 3) $R$ is torsion free $K$-module (that is $\lambda a=0$ implies $\lambda =0 $ or $a=0$).\\
 Define $R_1=R \oplus K$ (the additive group direct sum), thus
 $(a, \lambda)=(b, \mu)$ means, by the definition that $a=b$ and $\lambda
 =\mu$, and addition in $R_1$, is defined by the formula $(a, \lambda)+(b, \mu)=(a+b, \lambda +
 \mu)$. Define  $(a, \lambda)(b, \mu)=(ab+ \mu a+ \lambda b, \lambda
 \mu)$, $\mu (a, \lambda)=(\mu a, \mu \lambda)$, $(a, \lambda)^*=(a^*, \lambda
 ^*)$. Evidently, $R_1$ is also a $*$-algebra over $K$, has unity
 element $(0, 1)$ and $R$ is a
 $*$-ideal in $R_1$.
  
  Berberian has given a partial solution to Problem 1 as follows.
\begin{theorem}[{\cite[Theorem 1, page 31]{Ber}}] \label{Th1}
	Let $R$ be a weakly Rickart $*$-ring. If there exists an involutory integral domain $K$ such that $R$ is a $*$-algebra over $K$  and it is a torsion-free $K$-module, then $R$ can be embedded in a Rickart $*$-ring with preservation of RP's. 
\end{theorem} 
After 1972, there was little progress made toward the solution of Problem 1. In [14], Thakare and Waphare provided partial solutions, where the condition on the underlying weakly Rickart $*$-rings was relaxed in two distinct ways. For the solution of this open problem, Berberian used the condition that $R$ is a torsion-free left K-module, where K is an integral domain. Thakare and Waphare offered another solution in which the torsion-free condition was replaced with a different condition. They established the following.

\begin{theorem}[{\cite[Theorem 2]{Tha3}}] \label{Th2}
	A weakly Rickart $*$-ring $R$ can be embedded into a Rickart $*$-ring, provided there exists a ring $K$ such that
	\begin{enumerate}
\item $K$ is an integral domain with involution,
	\item  $R$ is $*$-algebra over $K$, and
	\item  For any $\lambda \in K-\{0\}$, there exist a projection $e_\lambda$ that is an upper bound for the set of left projections of the right annihilators of $\lambda$, that is if $x\in R$ and $\lambda x=0$ then $LP(x)\leq e_\lambda$.
		\end{enumerate}
\end{theorem}

  Based on the theory developed in Section 2, we pose the following problem for generalized Rickart $*$-rings, similar to Problem 1.
   
{\bf Problem 2}: Can every generalized weakly Rickart $*$-ring be embedded in a generalized Rickart $*$-ring? with preservation of $GRP$.

For a partial solution to Problem 2, the following results are required.

\begin{proposition} If $(a,0)\in  R_1 = R\oplus K$ then $(a,0)^n=(a^n ,0)$ for all  $n\in\mathbb N$.
\end{proposition}
\begin{proof} We prove the result by using mathematical induction on $n$. 
	Clearly result holds for $n=1$. Suppose result is true for $n=k$.
	That is $(a,0)^k = (a^k,0) $.
	Consider $(a,0)^{k+1}=(a,0)(a,0)^k=(a,0)(a^k,0)=(aa^k,0)=(a^{k+1},0)$.
	Thus by method of induction $(a,0)^n=(a^n,0)$ for all $n\in\mathbb N$. 
\end{proof}

\begin{proposition} If $(a,\lambda)\in R_1$ and $n \in \mathbb N$ then
	$(a,\lambda)^n=(a^n+\binom{n}{1} a^{n-1}\lambda+\binom{n}{2} a^{n-2}\lambda^2+\cdots+\binom{n}{n-1}a\lambda^{n-1} , \lambda^n)$.	
\end{proposition}
\begin{proof} We prove the result by using mathematical induction on $n$.
Clearly result is true for $n=1$. Suppose the result is true for $n=k$.
That is $(a,\lambda)^k=(a^k+\binom{k}{1}a^{k-1}\lambda+\cdots+\binom{k}{k-1}a\lambda^{k-1},\lambda^k)$.
   	Consider $(a,\lambda)^{k+1}=(a,\lambda)^k(a,\lambda)
	=(a^k+\binom{k}{1}a^{k-1}\lambda+\cdots+
	\binom{k}{k-1}a\lambda^{k-1} ,\lambda^k) (a,\lambda) $
	$=(a^{k+1}+\binom{k}{1}a^k\lambda+\cdots+\binom{k}{k-1}a^2\lambda^{k-1}+a\lambda^k
	+a^k\lambda+\binom{k}{1}a^{k-1}\lambda^2+\cdots+\binom{k}{k-1}a\lambda^k,\lambda^{k+1})
	=(a^{k+1}+[\binom{k}{0}+\binom{k}{1}]a^k\lambda+\cdots+[\binom{k}{k-1}+\binom{k}{k}]a\lambda^k,\lambda^{k+1})
	=(a^{k+1}+\binom{k+1}{1}a^k\lambda+\binom{k+1}{2}a^{k-1}\lambda^2+\cdots+\binom{k+1}{k}a\lambda^k,\lambda^{k+1})$.	Hence by induction, $(a,\lambda)^n=(a^n+\binom{n}{1}a^{n-1}\lambda+\binom{n}{2}a^{n-2}\lambda^2+\cdots+\binom{n}{n-1}a\lambda^{n-1},\lambda^n)$ for all $n\in\mathbb N$.
\end{proof}

\begin{lemma} If a ring $R$ has weakly proper involution then the involution on $R_1$ is weakly proper.
\end{lemma}
\begin{proof} Since the involution in $R$ is weakly proper.
	Therefore for $x\in R$, $xx^*=0 $ implies $x^n=0$ for some $n\in\mathbb N$.
	Let $(a,\lambda) \in R_1$ and  $(a,\lambda)(a,\lambda)^*=0$.
	This gives  $(a,\lambda)(a^*,\lambda^*)=0$.
	Therefore $(aa^*+\lambda^* a+\lambda a^*,\lambda\lambda^*)=0$.
	Since $K$ is an integral domain,  $\lambda\lambda^*=0$ implies $\lambda=0$.
	Hence $(aa^*,0)=0$. Therefore $aa^*=0$. Thus $a^n=0$ for some $n\in\mathbb N$.
	Hence $(a ,\lambda)^n=(a,0)^n=(a^n,0)=(0,0)=0$.
	So $R_1$ has weakly proper involution.
\end{proof}

  In the following result, we provide a partial solution to Problem 2.

\begin{theorem}  A generalized weakly Rickart $*$-ring can be embedded in a generalized Rickart $*$-ring
provided there exists a ring $K$ such that
\begin{enumerate}
	\item  $K$ is an integral domain with involution.
	\item  $R$ is $*$-algebra over $K$. 
	\item  For any nonzero $\lambda\in K$, there exists a projection $e_\lambda$ such that $\lambda x=0$ implies   $GRP(x)\leq e_\lambda$.
\end{enumerate}
\end{theorem}
\begin{proof} Let $ R_1 = R\oplus K$ (the additive group direct sum) with operations as defined above. First we prove that for any self-adjoint element $a\in R$ and $0 \ne \lambda \in K$ there exists largest projection $g$ such that $(ag+\lambda g)^m=0$ for some $m\in\mathbb N $. Let $GRP(a)=e_0$.
	Then  $a^me_0=a^m$ and $a^my=0$ implies $e_0y=0$ for some $m\in\mathbb N$. 
	Let $e_\lambda$ be a projection which exists by the assumption (3).
	Let $e$ be the largest projection in $\{e_0,e_\lambda\}$.
	Let $GRP(ae+\lambda e)=h$. Hence there exists $ m\in\mathbb N$ such that $(ae+\lambda e)^mh=(ae+\lambda e)^m$ and  $(ae+\lambda e)^my=0$ implies $hy=0$.  Now $e_0\leq e$ implies $e_0=e_0e=ee_0$.
	Let $g = e-h$.	Since $(ae+\lambda e)^me=(ae+\lambda e)^m$, we have $h\leq e$, that is $h=he =eh$. Therefore $eg=e-eh=e-h=g$. This gives $g\leq e$ and hence $g=eg=ge$.
	Thus $(ag+\lambda g)^m=(a eg+\lambda eg)^m=(ae+\lambda e)^mg
	=(ae+\lambda e)^m(e-h)=(ae+\lambda e)^m-(ae+\lambda e)^mh=0$.
	To prove $g$ is largest. Suppose $(ak+\lambda k)^m=0$. We have $e_0=e_0e=ee_0$. Since $a^me_0=a^m$, $a^me_0e=a^me$, which implies $a^me_0=a^me$. Therefore $ea^m=a^m$, which gives $kea^mk=ka^mk$. Hence $(ke-k)a^mk=0$. Since $(ak+\lambda k)^m=0$, we have $(ke-k)\{-\binom{m}{1}a^{m-1}k\lambda k-\cdots-\lambda^mk\}=0$. Equating coefficient of $a^{m-1}k$, we get  $\lambda m (ke-k)=0$. 
	Therefore $\lambda (ke-k)=0$. Let $GRP(ke-k)=f$. Then $(ke-k)^nf=	(ke-k)^n $ and $(ke-k)^ny=0$ implies  $fy=0$. Since $(ke-k)^ne=0$, we have $fe=0$. Further $\lambda (ke-k)=0$ implies $\lambda (ke-k)^n=0$.  Which gives $\lambda (ke-k)^nf=0$. Hence $ (ke-k)^n\lambda f=0$ . Therefore $f(\lambda f)=0$ . Thus $\lambda f=0$. By (3) $GRP(f)\leq e_\lambda\leq e$. Therefore $f\leq e$, that is $f=fe=ef$ (since $GRP(f)=f$). Hence $f=0$. As $(ke-k)^n= (ke-k)^nf$ implies $(ke-k)^n=0$. But $(ke-k)^n=\pm (ke-k)$. Hence $\pm (ke-k) =0$ gives $ke=k$. So $(ae+\lambda e)^m k=(aek+\lambda ek)^m=(ak+\lambda k)^m=0$. Therefore $hk=0$. Hence $kg=k(e-h)=ke-kh=k-0=k$. Thus $k\leq g$. Hence $g$ is the largest projection  such that $(ag+\lambda g)^m=0$.  Since $(0,1)$ is the unity element of $R_1$. By above Proposition \ref{s2pr8}  it is enough to show that $R_1$ is a generalized weakly Rickart $*$-ring. Let $(a,\lambda)\in R_1$.\\
	Suppose $\lambda=0$. Since $(a,0)\in R_1$, we have $a\in R$. 
	As $R$ is a generalized weakly Rickart $*$-ring, $GRP(a)=e$ exists in $R$.
	That is for some $n\in \mathbb N$, $a^ne=a^n$; and for $y \in R$, $a^ny=0$ implies $ey=0$.
	We will prove that $GRP(a,0)=(e,0)$. 
	Since $(a,0)^n(e,0)=(a^n,0)(e,0)=(a^ne,0)=(a^n,0)=(a,0)^n$.
	Suppose $(a,0)^n(b,\mu)=0$. Then $(a^n,0)(b,\mu)=0$.
	This implies $(a^nb+\mu a^n,0)=0$, that is $ a^nb+\mu a^n=0$.
	This gives $a^nb+\mu a^ne=0$, and hence $a^n (b+\mu e)=0$.
	Since  $a^ny=0$ implies  $ey=0$, we have $e(b+\mu e)=0$.
	Therefore $eb+\mu e=0$. That is $(e,0)(b,\mu)=0$. Hence $GRP(a,0)=(e,0)$.\\
	Now, suppose $\lambda\ne 0$. Then there exists a largest projection $g$ in $R$ such that $(ag+\lambda g)^t=0$ for some $t\in\mathbb N$. Note that $(-g,1)$ is a projection in $R_1$.
	We prove that $GRP(a,\lambda)=(-g,1)$.
	Consider $(a,\lambda)^t(-g,1)=(a^t+\binom {t}{1}a^{t-1}\lambda+\cdots+\binom {t}{t-1}a\lambda ^{t-1},\lambda^t)(-g,1) =(-a^tg-\binom {t}{1}a^{t-1}\lambda g-\cdots -\binom {t}{t-1}a\lambda^{n-1}g-\lambda^tg+a^t+\binom {t}{1}a^{t-1}\lambda+\cdots+\binom {t}{t-1}a\lambda^{t-1},\lambda^t)=(-(ag+\lambda g)^t+a^t+\binom {t}{1}a^{t-1}\lambda+\cdots+\binom {t}{t-1}a\lambda^{t-1},\lambda^t)=(a^t+\binom {t}{1}a^{t-1}\lambda+\cdots+\binom {t}{t-1}a\lambda^{t-1},\lambda^t)=(a,\lambda)^t$.
	To prove $(a,\lambda)^t(b,\mu)=0$ implies $(-g,1)(b,\mu)=0$. 
	Let $(a,\lambda)^t(b,\mu)=0$. Then $(a^t+\binom {t}{1}a^{t-1}\lambda+\cdots+\binom {t}{t-1}a\lambda^{t-1},\lambda^t)(b,\mu)=0$. 
	This implies $ \lambda^t\mu=0\Rightarrow\mu=0$  (since $\lambda\ne 0$). 
	Therefore  $(a^t+\binom {t}{1}a^{t-1}\lambda+\cdots+\binom {t}{t-1}a\lambda^{t-1},\lambda^t)(b,0)=0$.
	
	 To prove  $(a,\lambda)^t(b,0)^m=0$ implies $(-g ,1)(b,0)^m=0$ for some $m \in \mathbb N$.
	 Let $GLP(b)=f$, then there exist $m \in \mathbb N$ such that 
	 $fb^m=b^m$; and for $y \in R$, $yb^m=0$ implies $yf=0$. 
	Therefore $(a,\lambda)^t(b,0)^m=0$. This implies that 
	$(a^t+\binom {t}{1}a^{t-1}\lambda+\cdots+\binom {t}{t-1}a\lambda^{t-1},\lambda^t)(b^m,0)=0$.
	Therefore $(\{a^t+\binom {t}{1}\lambda a^{t-1}+\cdots+\binom {t}{t-1}\lambda^{t-1}a\}b^m+\lambda^tb^m,0)=0$.
	Hence $(\{a^tf+\binom {t}{1}\lambda a^{t-1}f+\cdots+\binom {t}{t-1}\lambda^{t-1}af+\lambda^tf\}b^m,0)=0$.
	This gives $(af+\lambda f)^tb^m=0$. 
	Since $yb^m=0$ implies $yf=0$, we have $(af+\lambda f)^tf=0$.
	This implies $(af+\lambda f)^t=0$.
	But $g$ is the largest projection such that $(ag+\lambda g)^t=0$.
	Therefore $f\leq g$, which gives  $f(1-g)=0$.
	Hence $(-g,1)(b,0)^m=(-g,1)(b^m,0)=(-gb^m+b^m,0)=((1-g)b^m,0)=((1-g)fb^m,0)=(0,0)=0$.
	
	Hence $GRP(a,\lambda)=(-g,1)$. Thus  $R_1$ is a  generalized Rickart $*$-ring.	
\end{proof}

\section{Parallelogram law, Generalized comparability and partial comparability}
         
   In this section, we prove that generalized Rickart $*$-ring satisfies parallelogram law. A sufficient condition is provided for a generalized Rickart $*$-ring to exhibit partial comparability. It is shown that a pair of projections in a generalized Rickart $*$-ring that satisfy the parallelogram law possesses orthogonal decomposition. 
          
  \begin{proposition}  Let $R$ be a generalized Rickart $*$-ring such that $GLP(x) \sim GRP(x)$ for all $x \in R$. Then $R$ satisfies the parallelogram law.
  	\end{proposition}
     \begin{proof}  Let $x=e - ef =e(1-f)$. Then $e\vee f=f + GRP(e(1-f))$.
     	Therefore $ GRP(e(1-f))=e \vee f-f$. Also $GLP(e(1-f))=e - e\wedge f$.
     	Since $GLP(x) \sim GRP(x)$, we have $GLP(e(1-f)) \sim GRP(e(1-f))$.
     	Therefore $ e-e\wedge f \sim e \vee f-f$. Thus $R$ satisfies the parallelogram law.
      \end{proof}
      	
      Projections $e$ and $f$ are said to be in a {\it position $p'$} in case $e \wedge (1-f)=0$ and $e \vee (1-f)=1$, that is $e$ and $1-f$ are complementary.
    
      \begin{proposition} \label{s3pr5} Let $R$ be a generalized Rickart $ *$-ring and $e,~f$ are projections in $R$. Then the following are equivalent. 
      	\begin{enumerate}
     \item $e,f$ are in position $p'$.
    \item $GLP(ef)=e$ and $GRP(ef)=f$.
     \end{enumerate}
    \end{proposition}
     \begin{proof} Suppose $e , f$ are in position $p'$.
     	Therefore $e\wedge (1-f)=0$ and $e\vee (1-f)=1$. Note that $ef=e(1-(1-f))$.
     	Since $e\vee f = f+GRP(e(1-f))$, we have
     	 $ GRP(ef)=GRP(e(1-(1-f)))=e\vee (1-f)-(1-f)=e\vee(1-f) - 1  + f = f$.
     	Similarly, $GL(ef) = GLP(e(1 - (1-f)))= e-e \wedge
     	(1-f)=e$. Conversely suppose $GLP(ef)=e$ and $GRP(ef)=f$.
     	Therefore $ GRP(e(1-(1-f)))= f$ implies $e \vee (1-f)-(1-f) = f$.
     	This gives $ e \vee (1-f) - 1+f = f$ that is $e \vee (1-f) = 1$.
     	Similarly $GLP(ef) = e$ implies $e \wedge (1-f) = 0$.
     	Thus $e ,f$ are in a position $p'$.
     \end{proof}
     
     \begin{proposition} Let $R$ be a generalized Rickart $*$-ring. Then the following are equivalent.
     	\begin{enumerate}
     \item $R$ satisfies the parallelogram law. 
     \item If $e, f$ are in position $p'$ then $e \sim f$.
    \end{enumerate}
     \end{proposition}
     \begin{proof} Suppose $R$ satisfies the parallelogram law.
     	Let $e, f$ be projections in a position $p'$.
     	Thus $ e \wedge (1-f) = 0$ and $e \vee (1-f) = 1$.
     	By the parallelogram law $e-e\wedge f \sim e\vee f-f$.
     	Replacing $f$ by $1-f$, we get $e - e\wedge(1-f) \sim e\vee(1-f) - (1-f)$.
     	Hence $e-0 \sim 1-1+f$, that is $e\sim f$.
     	Conversely, suppose $e,f$ are in a position $p'$ implies $e \sim f$.
     	Let $e,f$ be a pair of projections.
       Let $GRP(ef) = f'$ and $GLP(ef) = e'$.
     Since $ef f = ef$, we have $f' \le f$ that is $f'=f'f$.
     Similarly, since $e ef = ef$, we have $e' \le e$ that is $e'=e'e$.
      Therefore $ef=e'(ef)f'=(e'e)ff'=e'f'$.
     	Thus $ GRP(e'f') = f'$  and $GLP(e'f') = e'$.
     By Proposition \ref{s3pr5}, $e'$ and $f'$ are in a position $p'$. 
     Therefore $e'$ $\sim$ $f'$.
     Note that $e\wedge f = e-GLP(e(1-f))$.
     	Replacing $f$ by $1-f$, we get $ e\wedge(1-f)= e - GLP(ef) = e - e'$.
     	Hence $e' = e-e\wedge(1-f)$. Similarly  $f' = e\vee(1-f)-1+f$. 
     Therefore	$e' \sim f'$, which gives $e-e\wedge(1-f) \sim e \vee  (1-f)-1+f = f-(1-e)\wedge f$.
     Hence $ e-e\wedge(1-f)\sim f-(1-e)\wedge f$.
     Thus $ R $ satisfies the parallelogram law.
     \end{proof}
   Projections $e$ and $f$ in a $*$-ring $R$ are said to be generalized comparable if there exists central projection $h$ such that $he\lesssim hf$ and    $(1-h)f\lesssim (1-h)e$.
    We say that $R$ has generalized comparability (GC) if every pair of  projections is generalized comparable.    
   Projections $e$ and $f$ in a $*$-ring $R$ are said to be very orthogonal if there exists central projection $h$ such that $he=e$ and $hf=0$
      
     \begin{proposition} If projection $e$ and $f$ are very orthogonal in a generalized  Rickart $*$-ring $R$, then
     	$e,f$ are orthogonal, $GRP(e) GLP(f)=0$ and $eRf=0$.     	
     \end{proposition}
     \begin{proof} Suppose $e$ and $f$ are very orthogonal.
     	Therefore their exists central projection $h$ such that $he =e$  and  $hf= 0$.
     	Hence $ ef = hef = ehf = 0$. Therefore $e$ and $f$ are orthogonal.
     	Further, $GRP(e)= e$ and $GLP(f)= f$. Hence $GRP(e)GLP(f)= ef = 0$.
     	Also, $eRf = heRf = eRhf = 0$.
      \end{proof}	
     	
     \begin{example} Orthogonal projections need not be very orthogonal. 
     	 In $\mathbb{Z}_{12}$, the projections $0,1,4,$ and $9$ are all central.
     Since $2\cdot6 = 0$, we have $2$ and $6$ are orthogonal.
     	But $h\cdot2 = 2$ and $h\cdot6 = 0$ does not hold for any central projection $h$ in $\mathbb{Z}_{12}$.
     	Therefore $2, 6$ are not very orthogonal.	                 
    \end{example}
     Projections $e$ and $f$ in a $*$-ring $R$ are said to be partially comparable if there exists nonzero projections $e_0$ and $f_0$ such that $e_0\leq e , f_0\leq f$ and $e_0\sim f_0$.
    	We say that $R$ has partial comparability (PC) if $eRf\neq 0$ implies $e,f$ are partially comparable.    
      A $*$-ring is said to have orthogonal GC if every pair of orthogonal projections is generalized comparable. 
   
     \begin{proposition} If $R$ is a generalized Rickart $ *$-ring  with $GRP(x) \sim GLP(x)$ for all $ x \in R$ then $R$ has $PC$.
     \end{proposition}
     \begin{proof} Let $e$ and $f$ be projections in $R$ such that $eRf \neq 0$.
     	Let $x = eaf \in eRf$ be such that $x \neq 0$.
     	Let $GRP(x)= f_{0}$ and $GLP(x)= e_{0}$.
     	Then there exists $ n \in$ $\mathbb N$ such that
     	 $x^{n}f_{0}=x^{n}$; and for $y \in R$,  $x^{n}y=0 $ implies $f_{0}y=0$.
     	Now $xf =x $ gives $x^{n}f=x^{n}$. Since $GRP(x)=f_{0}$, we have $f_{0} \le  f$.
     	Similarly $e_{0}\le e$.
     	As $GLP(x) \sim GRP(x)$, we have  $e_{0} \sim f_{0}$.
     	Therefore there exist $e_{0},f_{0}$ such that $e_{0} \le e$, $f_{0} \le f$ and $e_{0} \sim f_{0}$.
     	Therefore $R$ has $PC$.
     \end{proof}
     
     \begin{example} Let $R=M_{2}(\mathbb{Z}_{3})$ and $e= 
     \begin{bmatrix}
     	1 & 0\\
     	0 & 0
     \end{bmatrix}$, $f=
     \begin{bmatrix}
     	2 & 2\\
     	2 & 2
     \end{bmatrix} \in R$.
     Let $A=\begin{bmatrix}
     	a & b\\
     	c & d
     \end{bmatrix} \in R$ be such that  $A^*A=e$  and  $AA^*=f$.
     Therefore $\begin{bmatrix}
     	a & c\\
     	b & d
     \end{bmatrix}$
     $\begin{bmatrix}
     	a & b\\
     	c & d
     \end{bmatrix}= \begin{bmatrix}
     	1 & 0\\
     	0 & 0
     \end{bmatrix}$ and $\begin{bmatrix}
     	a & b\\
     	c & d
     \end{bmatrix} \begin{bmatrix}
     	a & c\\
     	b & d
     \end{bmatrix} =\begin{bmatrix}
     	2 & 2\\
     	2 & 2
     \end{bmatrix}$. This gives     
     $a^{2}+c^{2} = 1,ab+cd = 0,b^{2}+d^{2} = 0$ and $a^{2}+b^{2} = 2, ac+bd = 2, c^{2}+d^{2} = 2$.
     If $a=0$ then $c=1$ or $2$. In both cases $~b=0,~d=0$. 
     If $a=1$ then $b=c=d=0$. 
     If $a=2$ then $b=0,~c=0,~d=0$.     
    Observe that none of the solution satisfy $a^{2}$+$b^{2}=2$.
    Hence $A^*A=e$ and  $AA^*=f$ do not hold for any $A \in R$.
    Therefore $ e \nsim f$.
       \end{example}
      	
     \begin{example} Let $R=M_{2}$($\mathbb{Z}_{3}$), 
      $e=\begin{bmatrix}
     	1 & 0\\
     	0 & 0
     \end{bmatrix}$,  $f=\begin{bmatrix}
     	2 & 2\\
     	2 & 2
     \end{bmatrix}$, $1-e=\begin{bmatrix}
     	0 & 0\\
     	0 & 1
     \end{bmatrix}$, $1-f=\begin{bmatrix}
     	2 & 1\\
     	1 & 2
     \end{bmatrix} \in R$.
     Note that $e(1-f)=\begin{bmatrix}
     	2 & 1\\
     	0 & 0
     \end{bmatrix}$ , $ef=\begin{bmatrix}
     	2 & 2\\
     	0 & 0
     \end{bmatrix}$ , $f(1-e)=\begin{bmatrix}
     	0 & 2\\
     	0 & 2
     \end{bmatrix}$ , $e(1-e)=\begin{bmatrix}
     	0 & 0\\
     	0 & 0
     \end{bmatrix}$, $f(1-f)=\begin{bmatrix}
     	0 & 0\\
     	0 & 0
     \end{bmatrix}$. 
     Hence $e , f , 1-e , 1-f$ are incomparable.
      For the pair $e , 1-f$, we have $e-e\wedge(1-f)= e$ and $e\vee(1-f)-1+f= f$.
       But $e \nsim  f$. Hence for $e , 1-f$ parallelogram law does not hold. 
     Note that $0 , 1$ are only central projections in $R$. Let $e , f$ as above. For $h=1, he \lnsim hf$ that is $e \lnsim f$, because if $e \lesssim f $ then $e \sim g \le f$ this implies $g=0$ or $f=0$ but $e \sim 0$ gives $e=0$, a contradiction and $ g = f$ gives $e \sim f$  a contradiction. Therefore $e \lnsim f$.
      For $h = 0 , (1-h)f \lnsim (1-h)e$ that is $f\lnsim  e$ as above. Thus $R$ does not have $GC$.
       We have $eRf \neq 0$ but $e , f$ do not have non-zero sub-projections $e_{0} , f_{0}$ such that  $e_{0} \sim f_{0}$. In fact $e$ and $f$ are only non-zero sub-projections of $e , f$ respectively but $e \nsim f$. Thus $R$ does not have $PC$.
      \end{example}
           	
     \begin{proposition} Let $R$ be a generalized Rickart $*$-ring satisfying parallelogram law. If $e,f$  are projections in $R$, then there exists orthogonal decomposition $e=e'+e''$, $f=f'+f''$ with $e' \sim f'$ and $ef''=fe''=0$.
     \end{proposition}	
     \begin{proof} Let $GLP(ef)=e'$ and $GRP(ef)= f'$. Then $ef= e'f'$. 
     Therefore 	$ GLP(e'f')= e'$ and $GRP(e'f')= f'$.
     By Proposition \ref{s3pr5}, we have $e'$ and $f'$ are in position $p'$.
     Therefore  $e' \sim f'$. Let $e''=e-e'$, $f''=f-f'$. Then  $e''(ef) = (e-e')ef = ef-e'ef = ef-ef = 0$.
     Therefore $(e''e)f = 0$. Since $e''\le  e$, we have  $e''f = 0 $. 
     Similarly $(ef)f''=0$ gives $ef'' = 0$.
     \end{proof}

    Disclosure statement: The authors report there are no competing interests to declare.\\

 \noindent {\bf Acknowledgment:}
              The authors are thankful to the anonymous referees for helpful comments and suggestions.

\end{document}